\documentclass{amsart}
\usepackage{amsfonts,amsmath,amstext,amssymb,amsthm, mathrsfs,amscd,enumerate,verbatim}
\usepackage[all]{xy}
\newtheorem{thm}{Theorem}[section]

\newtheorem{lem}[thm]{Lemma}

\newtheorem{rem}[thm]{Remark}

\theoremstyle{definition}

\numberwithin{equation}{section}
\newcommand{\ann}{\operatorname{Ann}}

\newcommand\supp {\operatorname{Supp}}

\newcommand\ass {\operatorname{Ass}}

\newcommand\fp{\mathfrak p}
\newcommand\fq{\mathfrak q}

\newcommand\N{\mathbb N}
\swapnumbers

\begin{document}

\title[characterization  of the  weakly Laskerian (FSF) modules]{A new characterization  of the  weakly Laskerian (FSF) modules}%
\author[A.~fathi]{Ali Fathi}
\address{Department of Mathematics, Zanjan Branch,
Islamic Azad University,  Zanjan, Iran.}
\email{alif1387@gmail.com}

%\date{*}%
%\dedicatory{}%
%\commby{*}%
% ----------------------------------------------------------------
\begin{abstract} Let $R$ be a commutative Noetherian ring and $M$ be an $R$-module such that the set of associated prime ideals of the quotient module $M/L$ is finite for all  submodules $L$ of $M$. In this paper, it is shown that there is a finitely generated submodule $N$ of $M$ such that the set of associated primes of $M/N$ and  the support of $M/N$ are equal.
 \end{abstract}
\maketitle
\setcounter{section}{0}
% ----------------------------------------------------------------
\section{\bf Introduction}

Throughout this paper, $R$ is a commutative Noetherian ring with nonzero identity.
 Let $N$ be a proper submodule of an $R$-module $M$. Then $N$ is called a {\it primary submodule} of $M$ when for all $r\in R$ and $x\in M$ if $rx\in N$, then $x\in N$ or $r^nM\subseteq N$ for some $n\in\N$. If $N$ is a primary submodule of $M$, then $\fp:=\surd(\ann_R(M/N))$
is a prime ideal of $R$ and $N$ is called a $\fp$-primary submodule of $M$.  An expression of $N$ as an intersection of finitely many primary submodules of $M$ is called a {\it primary decomposition} of $N$ in $M$.
Such a  primary decomposition
 $$N=M_1\cap\dots \cap M_n\quad \textrm {with } M_i\   \textrm{ $\fp_i$-primary in } M\ (1\leq i\leq n)$$
of $N$ in $M$ is said to  be minimal when  $\fp_1,\dots,\fp_n$ are distinct and $\bigcap_{1\leq j\neq i\leq n}M_j\nsubseteq M_i$ for all $1\leq i\leq n$. In this situation, we have $\ass_R(M/N)=\{\fp_1,\dots, \fp_n\}$; see \cite{m} for more details. An $R$-module $M$ is said to be {\it Laskerian} if each proper submodule of $M$ has a primary decomposition in $M$.

 Divaani--Aazar and Mafi in \cite{dm} defined an $R$-module $M$ to be {\it weakly Laskerian}, if $\ass_R(M/N)$ is a finite set for each submodule $N$ of $M$. Therefore Laskerian modules are weakly Laskerian.

On the other hand, Quy in \cite{q} introduced the class of FSF modules. An $R$-module $M$ is called an FSF module if there is a \textbf{F}initely generated submodule $N$ of $M$ such that the \textbf{S}upport of the
quotient module $M/N$ is \textbf{F}inite. It is easy to see that an FSF module is weakly Laskerian. Bahmanpour in \cite{b} proved that the converse statement is also true and so the class of FSF $R$-modules and the class of weakly Laskerian $R$-modules are equal. In this short note, we improve this result by proving  that if $M$ is weakly Laskerian, then $\supp_R(M/N)=\ass_R(M/N)$ for some finitely generated submodule $N$ of $M$ (and so $M$ is an FSF module).

\section{\bf Main Result}

\begin{thm}Let $M$ be an $R$-module. Then the following statements are equivalent:
\begin{enumerate}[\rm(i)]
\item $M$ is a weakly Laskerian $R$-module;
\item there is a finitely generated submodule $N$ of $M$ such that $\supp_R(M/N)=\ass_R(M/N)$ and $\supp_R(M/N)$ is a finite set;
\item $M$ is an FSF $R$-module.
\end{enumerate}
\end{thm}
\begin{proof}
The implication (ii)$\Rightarrow$(iii) is clear. To prove the implication (iii)$\Rightarrow$(i), assume that $N$ is a finitely generated submodule  of $M$ such that $\supp_R(M/N)$ is finite. Suppose that $L$ is an arbitrary submodule of $M$. It follows from the exact sequence
$$0\rightarrow(N+L)/L\rightarrow M/L\rightarrow M/(N+L)\rightarrow 0$$
that $\ass_R(M/L)\subseteq\ass_R((M/N)/((N+L)/N))\cup\ass_R(N/(N\cap L))$. Therefore  $\ass_R(M/L)$ is a finite set and consequently $M$ is weakly Laskerian.

(i)$\Rightarrow$(ii). Assume that $M$ is weakly Laskerian and to prove (ii) it is sufficient for us to show that  $\supp_R(M/N)=\ass_R(M/N)$ for some finitely generated submodule $N$ of $M$.

We set  $M_0:=0$ (the zero submodule of $M$). Assume $i\in\N$.  If $\supp_R(M/M_{i-1})=\ass_R(M/M_{i-1})$, then we set $N:=M_{i-1}$ and we end the process.  Otherwise,  we construct $x_i, M_i,  \fq_i, \fp_i, \Sigma_i$ as follows. We set $\Sigma_i:=\supp_R(M/M_{i-1})\setminus\ass_R(M/M_{i-1})$. Since $\Sigma_i$ is a non-empty set of ideals of $R$ and $R$ is Noetherian,  $\Sigma_i$ has a maximal element under inclusion, say  $\fp_i$.
  There exist $\fq_i\in\ass_R(M/M_{i-1})$ and $x_i\in M$ such that $\fq_i=(0:_Rx_i+M_{i-1})\subset\fp_i$ (note that $\fq_i\in\ass_R(M/M_{i-1})$ while $\fp_i\notin\ass_R(M/M_{i-1})$ and so $\fq_i\neq\fp_i$). Now, we  set $M_i:=M_{i-1}+\fp_ix_i$. It is clear that $M_i$ is a finitely generated submodule of $M$. We claim that after a finite number of steps the procedure must stop.  Assume for the sake of contradiction that $\supp_R(M/M_i)\neq\ass_R(M/M_i)$ for all $i$ and $x_i, M_i,  \fq_i, \fp_i, \Sigma_i$ are constructed for all $i\in\N$ as above. Before continuing the proof, we need the following two lemmas.

  \begin{lem}\label{lem1} For each $ j\in \N$, $\ass_R(M_j/M_{j-1})=\{\fq_j\}$.
  \end{lem}
  \begin{proof}
  Since $\fp_j\nsubseteq\fq_j$, we obtain $M_{j-1}\subset M_j$ and so $\ass_R(M_j/M_{j-1})\neq\emptyset$. Now, suppose that  $\fp\in\ass_R(M_j/M_{j-1})$. Then $\fp=(0:_Ra_jx_j+M_{j-1})$ for some $a_j\in\fp_j$. Therefore
   $\fp a_j\subseteq\fq_j$ and so $\fp\subseteq \fq_j$. The reverse inclusion is clear and hence  $\fp=\fq_j$.
  \end{proof}

  \begin{lem} If  $1\leq i\leq j$, then $\fp_i=(0:_Rx_i+M_j)\in\ass_R(M/M_j)$ and $\fp_i\nsubseteq\fp_j$ whenever $1\leq i<j$.
  \end{lem}
  \begin{proof}
  First, we prove that $\fp_j=(0:_Rx_j+M_j)$ for all $j\in\N$. Assume that $r\in(0:_Rx_j+M_j)$. Therefore $rx_j=m_{j-1}+a_jx_j$ for some $m_{j-1}\in M_{j-1},a_j\in\fp_j$. It follows that $r-a_j\in\fq_j$ and so $r\in\fp_j$. This shows that $(0:_Rx_j+M_j)\subseteq\fp_j$. The reverse inclusion is clear and so $\fp_j=(0:_Rx_j+M_j)\in\ass_R(M/M_j)$ for all $1\leq j$.

  Next, we show that    if $1\leq i<j$, then $\fp_i\nsubseteq\fp_j$. Assume for the sake of contradiction that $\fp_i\subseteq\fp_j$ for some $i, j$ with $i<j$. We can assume that $j$ is the least positive integer with the property  that there exists $i$ with $i<j$ such that $\fp_i\subseteq\fp_j$. If $\fp_i=\fp_j$, then since $\fp_i=(0:_Rx_i+M_i)$ we have $\fp_j\in\ass_R(M/M_i)$ while $\fp_j\notin\ass_R(M/M_{j-1})$ by definition. Therefore $i<j-1$ and  there exists $i+1\leq k< j$ such that $\fp_j\in\ass_R(M/M_{k-1})$ and $\fp_j\notin\ass_R(M/M_{k})$. Also, if $\fp_i\subset \fp_j$,  then, by the maximality of $\fp_i$ in $\Sigma_i$,  $\fp_j\in\ass_R(M/M_{i-1})$ while  $\fp_j \notin\ass_R(M/M_{j-1})$ by definition. Therefore there exits $i\leq k <j$ such that $\fp_j\in\ass_R(M/M_{k-1})$ and $\fp_j\notin\ass_R(M/M_k)$. In both cases,
  we deduce from the exact sequence
 $$0\rightarrow M_k/M_{k-1}\rightarrow M/M_{k-1}\rightarrow M/M_k\rightarrow 0$$
 that $\fp_j\in\ass_R(M_k/M_{k-1})$ and so Lemma \ref{lem1} implies that $\fp_j=\fq_k\subset\fp_k$. Since $\fp_i\subseteq\fp_j$, we obtain $\fp_i\subset\fp_k$, which is impossible by the minimality of $j$. Therefore $\fp_i\nsubseteq\fp_j$ for all $1\leq i<j$.

 Finally, we prove by induction on $j$  that $\fp_i=(0:_Rx_i+M_j)$ for all $1\leq i\leq j$. If $j=1$, then we have $\fp_1=(0:_Rx_1+M_1)$. Now assume that $j>1$ and the result has been proved for the smaller values of $j$.  The case $i=j$ is proved at the beginning of the proof of this lemma. So assume  that $i<j$. By the above proof, we have $\fp_i\nsubseteq \fp_j$ and hence there exists $s\in \fp_i\setminus\fp_j$. Suppose that $r\in(0:_Rx_i+M_j)$. Therefore $rx_i=m_{j-1}+a_jx_j$ for some $m_{j-1}\in M_{j-1}, a_j\in\fp_j$. Hence
$srx_i=sm_{j-1}+sa_jx_j$ and so $sa_jx_j\in M_{j-1}$. It follows that $sa_j\in\fq_j$ and so $a_j\in\fq_j$. Therefore $a_jx_j\in M_{j-1}$ and so  $rx_i\in M_{j-1}$. By the inductive hypothesis, we have $(0:_Rx_i+M_{j-1})=\fp_i$ and hence $r\in\fp_i$. Since $r$ is an arbitrary element of $(0:_Rx_i+M_j)$, we obtain $(0:_Rx_i+M_j)\subseteq\fp_i$. The revers inclusion is clear. Therefore  $(0:_Rx_i+M_j)=\fp_i$, as required. This completes the proof of the lemma.
   \end{proof}
 Now, we continue the proof of the theorem.  We set $M_\infty:=\bigcup_{i=1}^\infty M_i$ (note that $M_1\subset M_2\subset M_3\subset \dots$). We show that $(0:_Rx_i+M_\infty)=\fp_i$ for all $i\in\N$. Assume $r\in (0:_Rx_i+M_\infty)$. Therefore $rx_i\in M_j$ for some sufficiently large positive integer $j$. We can assume that $i\leq j$ and so $r\in (0:_Rx_i+M_j)=\fp_i$ by above lemma.  This shows that $(0:_Rx_i+M_\infty)\subseteq\fp_i$. The reverse inclusion is clear and so, in view of the above lemma, $\ass_R(M/M_\infty)$ contains the distinct prime ideals $\fp_1, \fp_2, \dots$, which is impossible because $M$ is weakly Laskerian. Therefore $\supp_R(M/M_i)=\ass_R(M/M_i)$ for some $i$, as required. This proves  the implication (i)$\Rightarrow$(ii) and so the proof of the theorem is completed.
\end{proof}
\begin{rem} In the published version of this paper (see, A. Fathi, A new characterization  of the  weakly Laskerian (FSF) modules, to appear in Commun. Korean Math. Soc.) \cite[Theorem 3.3]{b} is used to prove the above theorem. But here we provide an independent proof.
\end{rem}

  \end{document}